\documentclass[12pt]{amsart}
\usepackage{graphicx}
\usepackage{caption}
\usepackage{amsmath,graphics}
\usepackage{amsfonts,amssymb}
\usepackage{xypic}
\usepackage{tikz} 
\usepackage{pgfplots}
\usepackage{comment}
\specialcomment{proofc}{}{}
\includecomment{proofc}
\theoremstyle{plain}
\newtheorem*{theorem*}{Theorem}
\newtheorem*{lemma*} {Lemma}
\newtheorem*{corollary*} {Corollary}
\newtheorem*{proposition*}{Proposition}
\newtheorem*{conjecture*}{Conjecture}
\newtheorem{theorem}{Theorem}[section]
\newtheorem{lemma}[theorem]{Lemma}
\newtheorem*{theorem1*}{Theorem 1}
\newtheorem*{theorem2*}{Theorem 2}
\newtheorem*{theorem3*}{Theorem 3}

\newtheorem{proposition}[theorem]{Proposition}

\newtheorem{ther}{Theorem}

\newcommand\rg{\mathrm{rg}\,}

\newcommand{\co}{\colon \thinspace}

\newtheorem{cor}[ther]{Corollary}
\theoremstyle{remark}
\newtheorem*{remark}{Remark}

\newtheorem{example*}{Example}
\newtheorem*{claim}{Claim}

\theoremstyle{definition}

\newcommand{\lmfrac}[2]{\mbox{\small$\displaystyle\frac{#1}{#2}$}}

\def\op{\operatorname}

\def\G{\Gamma}

   \def\Z{\Bbb{Z}} \def\R{\Bbb{R}} 
\def\N{\Bbb{N}}   \def\ll{\langle} \def\rr{\rangle}
 \def\a{\alpha}   \def\bp{\begin{pmatrix}}
 \def\ep{\end{pmatrix}} \def\bn{\begin{enumerate}} 
   \def\en{\end{enumerate}}
\def\ba{\begin{array}} \def\ea{\end{array}}  
\def\hom{\op{Hom}}
  \def\a{\alpha}  
\def\id{\mbox{id}}   
  \def\Ker{\mbox{Ker}}
\def\ker{\mbox{Ker}}\def\be{\begin{equation}} \def\ee{\end{equation}} 
   
 \def\hom{\mbox{Hom}}

\def\zt{\Z[t^{\pm 1}]}    \def\rk{\op{rk}}

\begin{document}
\title{Rank gradients of infinite cyclic covers of K\"ahler manifolds}

\author{Stefan Friedl}
\address{Fakult\"at f\"ur Mathematik, Universit\"at Regensburg, Germany}
\email{sfriedl@gmail.com}

\author{Stefano Vidussi}
\address{Department of Mathematics, University of California,
Riverside, CA 92521, USA} \email{svidussi@ucr.edu} 
\date{\today}

\begin{abstract} Given a K\"ahler group $G$ and a primitive class $\phi\in H^1(G;\Z)$, we show that the rank gradient of $(G,\phi)$ is zero if and only if $\ker~ \phi \leq G$ finitely generated. Using this approach, we give a quick proof of the fact (originally due to Napier and Ramachandran) that K\"ahler groups are not properly ascending or descending HNN extensions. Further  investigation of the properties of Bieri--Neumann--Strebel invariants of K\"ahler groups allows us to show that a large class of groups of orientation--preserving PL homeomorphisms of an interval (customarily denoted $F(\ell,\Z[\frac{1}{n_1\cdots n_k}],\langle n_1,\dots,n_k\rangle)$), which generalize Thompson's group $F$, are not K\"ahler.
\end{abstract}


\maketitle

\section{Introduction}

In \cite{Lac05 } M. Lackenby introduced the notion of \textit{rank gradient} $\rg(G,\{ G_i \})$ associated to a \textit{lattice}  $\{ G_i \}$ of finite index subgroups $G_{i} \leq_{f} G$ (i.e. a collection of subgroups which is closed with respect to intersection). This rank gradient is defined as \begin{equation} \label{eq:rg} \rg(G,\{ G_i \}) =  \liminf_{i\to \infty} \frac{\rk(G_i)}{[G:G_{i}]}, \end{equation}
where given a group $H$ we denote by $\rk(H)$ the rank of $H$, i.e.\ the minimal number of generators.
 Lackenby computed the rank gradient for various classes of groups and he illustrated  its r\^ole in understanding properties of the lattice and, in turn, of the group $G$ itself. Successive work of many authors has related this notion with various invariants of the group $G$.

The explicit calculation of the rank gradient appears rarely straightforward, and usually requires some relatively sophisticated information on the group. 

Here we will focus on the rank gradient defined by lattices associated to the infinite cyclic quotients of $G$, namely collections of the type \[ \{G_i=\ker(G \xrightarrow{\phi}\Z\to \Z_i)\} \] determined by a primitive class $\phi \in H^{1}(G;\Z) =\hom(G,\Z)$; this rank gradient will be denoted, for simplicity, $\rg(G,\phi)$.

One example where this calculation can be carried out in some detail is that of the fundamental group $\pi = \pi_{1}(N)$ of a $3$--manifold $N$ with $b_{1}(N) \geq 1$:  in \cite{DFV14} the authors proved that $\rg(\pi,\phi)$ vanishes if and only if $\phi$ is represented by a locally--trivial fibration of $N$ over $S^1$ with surface fibers. More  precisely $\rg(\pi,\phi) = 0$ if and only if  $\ker~ \phi \leq \pi$ is finitely generated, which is equivalent  to what is stated above by virtue of Stallings' Fibration  Theorem (see \cite{S62}). While one direction of the proof is straightforward, the other originates from some rather deep properties of $3$--manifold groups, which can be interpreted as a consequence of the separability results of Agol \cite{Ag13} and Wise \cite{Wi12}.

In this paper, we will show that a similar result holds for fundamental groups of K\"ahler manifolds, or K\"ahler groups.
 (The standard reference for the study of this class of groups is \cite{ABC$^+$96}; more recent results are surveyed in \cite{Bu11}.) Much like in  the case of $3$--manifold groups, the non--straightforward direction of the proof of our main theorem will come from suitably exploiting nontrivial properties of K\"ahler groups. Interestingly enough, the key tool here will also be separability, this time for the fundamental groups of surfaces. 

For sake of notation, we define a \textit{K\"ahler pair} to be the pair $(G,\phi)$ where $G$ is a K\"ahler group and $\phi \in H^{1}(G;\Z) = \hom(G,\Z)$ is a primitive (i.e. surjective) cohomology class. 
Our main theorem is the following:

\begin{ther} \label{thm:mainproof}
Let $(G,\phi)$ be a  K\"ahler pair. Then $\rg(G,\phi) = 0$ if and only if $\ker~\phi \leq G$ is finitely generated.
\end{ther}

Note that for K\"ahler groups there is no analog (to the best of our understanding) to Stallings' Fibration Theorem, for a class with $\rg(G,\phi) = 0$; to start, there is no reason to expect that the kernel $\ker~\phi$ is {\em finitely presented}, as K\"ahler groups, in general, are not coherent. Instead, in this case it is the classes $\phi \in \hom(G,\Z)$ with $\rg(G,\phi) > 0$ that admit a geometric interpretation; in fact, as we discuss in Section  \ref{sec:rg} (equation \ref{eq:pencil} and following remarks) these are the classes that (virtually)
factorize through the epimorphism induced, in homotopy, by a genus $g \geq 2$ pencil on a K\"ahler manifold $X$ such that $G = \pi_1(X)$. In fact,  as we will see in the proof of Theorem \ref{thm:mainproof}, the existence of that pencil (at homotopy level) can be thought of as the very reason why for these classes we have $\rg(G,\phi) > 0$.

As a corollary of this theorem, we will give a quick proof of a result of \cite{NR08}, Theorem 0.3(ii):

\begin{cor} \label{cor:ascending} Let $G$ be a K\"ahler group; then $G$ cannot be a properly ascending or descending HNN extension of a finitely generated subgroup. \end{cor}

The latter result was originally proven by Napier and Ramachandran in order to show that Thompson's group $F$ (and some of its generalizations) are not K\"ahler, answering a question originally posed by Ross Geoghegan.

Several classes of generalizations of the group $F$ have been studied. We will be interested here in a rather broad collection of finitely presented subgroups, usually denoted as $F(\ell,\Z[\frac{1}{n_1\cdots n_k}],\langle n_1,\dots,n_k\rangle)$, of the group of orientation--preserving PL homeomorphisms of the interval $[0,\ell]$ that was first investigated by Brown in \cite{Bro87a} (and in unpublished work), and further discussed in \cite{St92}; see Section \ref{stein} for the definition. This collection of groups (which for notational simplicity will be referred to here as generalized Thompson groups) include Thompson's group $F = F(1,\Z[\frac{1}{2}],\langle2\rangle)$. However, of relevance to  us, for ``most" of them Corollary \ref{cor:ascending} yields no information to decide if they are K\"ahler. Using Delzant's work \cite{De10} on Bieri--Neumann--Strebel invariants (\cite{BNS87}) of K\"ahler groups, we prove the following:

\begin{ther} \label{thm:st} Let $G = F(\ell,\Z[\frac{1}{n_1\cdots n_k}],\langle n_1,\dots,n_k\rangle)$ be a generalized Thompson group;  then $G$ is not K\"ahler. \end{ther}

The structure of the paper is the following. Section \ref{sec:rg} is devoted to the proof of Theorem \ref{thm:mainproof} and  Corollary \ref{cor:ascending}. 
In Section \ref{stein} we prove that the generalized Thompson groups are not K\"ahler (Theorem \ref{thm:st}) and we also discuss the fact that for ``most" of these groups Corollary \ref{cor:ascending} (even applied to finite index subgroups) is inconclusive. Finally in Section \ref{sec:rgbns} we discuss some examples to illustrate the intricacy of the relation between the rank gradient and the BNS invariant of groups in general. The Appendix is devoted to the proof of some basic results on generalized Thompson groups used in Section \ref{stein}.

\subsection*{Acknowledgments.} The authors would like to thank the referee for carefully reading the manuscript and providing us with useful comments. The first author was supported by the SFB 1085 `higher invariants' funded by the Deutsche Forschungsgemeinschaft DFG.

\section{Rank Gradient of K\"ahler Groups} \label{sec:rg}

We start with two lemmata. The first is straightforward.

\begin{lemma} \label{lemma:fib} Consider the short exact sequence of finitely generated groups \[ 1 \to K \to G \xrightarrow{p} \Gamma \to 1. \] Let $\rho \in H^{1}(\Gamma;\Z)$ be a primitive class; then $p^{*} \rho \in H^1(G;\Z)$ is primitive and $\rg(G,p^{*} \rho) = \rg(\Gamma,\rho)$.   \end{lemma}
\begin{proof} The fact that $p^{*} \rho = \rho \circ p$ is primitive is obvious from the fact that $p$ is surjective, hence so is  $p^{*} \rho$ thought of as an element of $Hom(G,\Z)$. The rank gradients in question are associated respectively to the lattices  \[ \{G_i=\ker(G \xrightarrow{p^{*} \rho}\Z\to \Z_i)\} \ \ \ \text{and} \ \ \ \{\Gamma_i=\ker(\Gamma \xrightarrow{\rho}\Z\to \Z_i)\}.\] These groups are related, for each $i \geq 1$, by a short exact sequence   \[ 1 \to K \to G_i \xrightarrow{p|_{G_{i}}} \Gamma_{i} \to 1 \] hence the inequalities \[ \rk(\Gamma_{i}) \leq \rk(G_{i}) \leq \rk(\Gamma_{i}) + \rk(K). \] As the second summand in the rightmost term is independent of $i$, a check of the definition in (\ref{eq:rg}) yields the desired equality. \end{proof}

The proof of the second lemma is more elaborate. Remember that a finitely presented group $\G$ is LERF if, given any finitely generated subgroup $A \leq \G$, and any element $\gamma \in \G \setminus A$, there exists an epimorphism $\a\colon \G \to Q$ to a finite group such that $\a(\gamma) \notin \a(A)$.

\begin{lemma} \label{lemma:virt} Let $\G$ be a finitely presented LERF group, and let $\rho \in \hom(\G,\Z) = H^1(\G;\Z)$ be a primitive class. Then the following are equivalent:
\bn
\item $\ker~ \rho \leq \G$ is infinitely generated;
\item there exists a normal finite index subgroup $\widetilde{\Gamma} \leq_{f} \Gamma$ and an epimorphism $\widetilde{q}\colon \widetilde{\Gamma} \to F_n$ to a free nonabelian group such that
$\rho|_{\widetilde{\Gamma}}\colon \widetilde{\Gamma} \to \Z$ factorizes through
$\widetilde{q}$, namely we have
 \[ \xymatrix{1 \ar[r] & \widetilde{\Gamma} \ar[r] \ar[dr]_{\widetilde{q}} & \Gamma \ar[r] \ar[dr]^{\rho} & Q \ar[r] & 1 \\ & & F_n \ar[r]_{\psi} & \Z. & } \]  with $Q$ a finite group. 
\en
\end{lemma}

\begin{proof} Let $\G$ be a finitely presented group and $\rho \in H^1(\G;\Z)$ a primitive class. We can write $\G$ as an HNN extension of a finitely generated base group $B \leq \ker ~ \rho$ with finitely generated associated subgroups $A_{\pm} \leq B$, namely we have \[ \G = \langle B,t| t^{-1}A_{+} t = A_{-}\rangle,\]  with $\rho(t) = 1 \in \Z$. Moreover, $\ker ~ \rho$ is finitely generated if and only if $A_{\pm} = B =\ker ~ \rho$ (see \cite{BS78}). 

Assume that $\ker ~ \rho$ is infinitely generated; then one (or both) of $A_{\pm}$ is a proper subgroup of $B$, say $A_{+} \lneq B$. Since $\G$ is LERF, for any element $b \in B \setminus A_{+}$ there exists an epimorphism to a finite group $\a\colon \G \to S$ such that $\a(b) \notin \a(A_{+})$. In particular, $\a(A_{+})$  is strictly contained in $\a(B)$.  We can now define an epimorphism from $\G = B_{*A}$ to the HNN extension $\a(B)_{*\a(A)}$ of the finite base group $\a(B)$ with finite associated (isomorphic) subgroups $\a(A_{\pm}) \lneq \a(B)$. (We use here standard shorthand notation for HNN extensions in terms of base and abstract associated subgroups.) Precisely, we have a map 
\[ q\colon  \langle B,t|t^{-1} A_{+} t = A_{-} \rangle \longrightarrow \langle \a(B),t|t^{-1} \a(A_{+}) t = \a(A_{-}) \rangle. \] 
Clearly the map $\rho\colon \G \to \Z$ descends to $\a(B)_{*\a(A)}$. By \cite[Prop~11~p.~120 and Exercise~3~p.~123]{Se80}
 this group admits a finite index free nonabelian normal subgroup, and we get the commutative diagram \[ \xymatrix{1 \ar[r] & \widetilde{\Gamma} \ar[r] \ar[d]_{\tilde{q}}  & \Gamma \ar[r] \ar[d]^{q} \ar[ddr]^{\rho} & Q \ar[r] \ar[d]^{\cong} & 1 \\ 1 \ar[r] & F_n \ar[r] \ar[drr]_{\psi} & \a(B)_{*\a(A)} \ar[r]|\hole \ar[dr] & Q \ar[r] & 1 \\ & &  & \Z & } \] which yields the virtual factorization of $\rho \in \hom(\Gamma,\Z)$.

To prove the other direction, assume that $\rho$ admits a virtual factorization as in $(2)$. To show that $(1)$ holds,  it is enough to prove that $\ker ~ \rho|_{\widetilde \G}$ is an infinitely generated subgroup of ${\widetilde \G}$, as the latter is finite index in $\G$. 
But the group $\ker ~ \rho|_{\widetilde \G}$ surjects to the kernel of a nontrivial homomorphism $F_{n} \to \Z$. This kernel must be infinitely generated, as finitely generated normal subgroups of a free nonabelian group are finite index (see e.g. \cite[Lemma 3.3]{Ca03}). 
\end{proof}

We are now in position to prove our main theorem, that we restate for sake of convenience.

\begin{theorem} 
Let $(G,\phi)$ be a  K\"ahler pair. Then $\rg(G,\phi) = 0$ if and only if $\ker~\phi \leq G$ is finitely generated.
\end{theorem}

\begin{proof} The proof of the ``if" direction is straightforward: if $\phi\colon G \to \Z$ is an epimorphism with finitely generated kernel $\ker~\phi$, and we denote by $\{g_{1},\dots,g_{k}\}$ a generating set for $\ker~\phi$, then the  group $G$ is the semidirect product $\ker~ \phi \rtimes \langle t \rangle$ and admits the generating set $\{g_{1},\dots,g_{k},t\}$, so that the rank of $G$ is bounded above by $k + 1$. It is well--known that \[ G_{i} :=\ker(G\xrightarrow{\phi}\Z\to \Z/i)\] equals $\ker~ \phi \rtimes \langle t^{i} \rangle$, hence $\mbox{rk}(G_i)$ is uniformly bounded above by $k+1$. This immediately entails the vanishing of $\rg(G,\phi)$.

For the ``only if" direction, we proceed as follows. Let  $\phi\colon G \to \Z$ be an epimorphism with infinitely generated kernel. By \cite[Theorem 4.3]{NR01}, such an epimorphism factorizes through an epimorphism $p\colon G \to \Gamma$ with finitely generated kernel $K$ where $\Gamma$ is a cocompact Fuchsian group of positive first Betti number, which arises as the fundamental group of a hyperbolic Riemann orbisurface $\Sigma^{orb}$ of positive genus:
\begin{equation} \label{eq:pencil}  \xymatrix{1 \ar[r] & K \ar[r] & G \ar[r]^{p} \ar[drr]_{\phi} & \Gamma \ar[r] \ar[dr]^{\rho} & 1\\ & & &  & \Z }\end{equation}
Although not strictly necessary for the proof, we can integrate this picture and interpret (following \cite{Ca03,NR06}) the maps of the diagram in terms of induced maps of  an irrational pencil $f\colon  X \to \Sigma$, where $X$ is a K\"ahler manifold with $G = \pi_{1}(X)$ and  $\Sigma$ a Riemann surface of positive genus. In the presence of multiple fibers, to obtain a surjection of fundamental groups with finitely generated kernel, we promote the fundamental group of $\Sigma$ to the fundamental group $\Gamma$ of the orbisurface $\Sigma^{orb}$ with orbifold points determined by the multiple fibers.

The proof now boils down to the use of the lemmata above and some basic results from \cite{Lac05 }. By Lemma \ref{lemma:fib},  $\rg(G,\phi) = \rg(G,p^{*} \rho) = \rg(\Gamma,\rho)$. By definition, the latter is the rank gradient $\rg(\Gamma,\{\Gamma_{i}\})$ of the lattice $\{\Gamma_i=\ker(\Gamma \xrightarrow{\rho}\Z\to \Z_i)\}$.  Rank gradients behave well under passing to finite index subgroup: let $\widetilde{\Gamma} \leq_{f} \Gamma$ be a finite index subgroup; \cite[Lemma 3.1]{Lac05 } asserts that $\rg(\Gamma,\{\Gamma_{i}\})$ is nonzero if and only if $\rg(\widetilde{\Gamma},\{\widetilde{\Gamma} \cap \Gamma_{i}\})$ is nonzero.
We can now apply Lemma \ref{lemma:virt} to $\Gamma$, as the latter is LERF by \cite{Sc78} and for any $\rho \in H^1(\G;\Z)$, $\ker ~ \rho$ is infinitely generated (see e.g. \cite[Lemma 3.4]{Ca03}). Keeping the notation from Lemma \ref{lemma:virt} we have
 \[ \xymatrix{ & \widetilde{\Gamma} \ar[r] \ar[dr]_{\widetilde{q}} & \Gamma \ar[dr]^{\rho} &  & \\ & & F_n \ar[r]_{\psi} & \Z & } \]
so  $\widetilde{q}\colon \widetilde{\Gamma} \to F_n$ restricts to an epimorphism from $\widetilde{\Gamma} \cap \Gamma_{i}$ to $\ker(F_n \xrightarrow{\psi}\Z\to \Z_i)\}$. 
It follows that we have the inequality \begin{equation}  \label{eq:rkin1} \rk(\widetilde{\Gamma} \cap \Gamma_{i}) \geq \rk (\ker(F_n \xrightarrow{\psi}\Z\to \Z_i)).\end{equation} We endeavor now to compute the term on the right hand side. In general, the map $\psi\colon F_n \rightarrow \Z$  identified in Lemma \ref{lemma:virt} may fail to be surjective, when $m$ is not primitive and has some divisibility $m \in \Z$. In that case the image of $\psi$ is a subgroup $m\Z \subset \Z$. A simple exercise shows that the index of $\ker(F_n \xrightarrow{\psi}\Z\to \Z_i)$ in $F_n$ equals $\frac{\op{lcm}(m,i)}{m}$. The Reidemeister--Schreier formula gives then the rank of this free group:
\begin{equation}  \label{eq:rkin2}  \rk (\ker(F_n \xrightarrow{\psi}\Z\to \Z_i)) = \frac{{\op{lcm}(m,i)}}{m}(n-1) + 1. \end{equation} Combining equations (\ref{eq:rkin1}) and (\ref{eq:rkin2}) in the calculation of the rank gradient we get
\[ \frac{\rk(\widetilde{\Gamma} \cap \Gamma_{i})}{[{\widetilde \Gamma}: \widetilde{\Gamma} \cap \Gamma_{i}]} \geq  \frac{\rk(\widetilde{\Gamma} \cap \Gamma_{i})}{[\Gamma: \Gamma_{i}]} \geq  \frac{1}{i}\Big( \frac{\op{lcm}(m,i)}{m}(n-1) + 1 \Big) \geq \frac{1}{m}(n-1) + \frac{1}{i}  \] from which it follows that \[ \rg(\widetilde{\Gamma},\{\widetilde{\Gamma} \cap \Gamma_{i}\}) \geq \frac{1}{m}(n-1) > 0.\] As observed above, this entails $\rg(G,\phi) > 0$. 
\end{proof}

\begin{remark} While there is no reason to expect that K\"ahler groups are LERF (or even satisfy weaker forms of separability, that would be still sufficient for Lemma \ref{lemma:virt} to hold) we see that the presence of an irrational pencil, in the form of $(\ref{eq:pencil})$ yields essentially the same consequences, using the fact that cocompact Fuchsian groups are LERF. In that sense, the proof above and the proof for $3$--manifold groups in \cite{DFV14} originate from similar ingredients.
\end{remark}

We can now give a proof of Corollary \ref{cor:ascending}. Recall that a finitely presented $G$ is an \textit{ascending} (respectively \textit{descending}) HNN extension if it can be written as \[ G = \langle B,t| t^{-1}A_{+} t = A_{-}\rangle \] where $B$ is  finitely generated and  where $A_{+} = B$ (respectively $A_{-} = B$).

\begin{proof} Let $G$ be a K\"ahler group that can be written as an ascending or descending  HNN extension with stable letter $t$, and denote by $\phi\colon G \to \Z$ the homomorphism sending $t$ to the generator of $\Z$ and the base of the extension (a subgroup $B \leq G$ with generating set $\{g_{1},\dots,g_{k}\}$) to $0$. We claim that $\rg(G,\phi)$ vanishes. The proof is almost \textit{verbatim} the proof of the ``if" direction of Theorem \ref{thm:mainproof}, once we observe (using Schreier's rewriting process) the well--known fact that $G_{i} :=\ker(G\xrightarrow{\phi}\Z\to \Z/i)$ is generated by $\{g_{1},\dots,g_{k},t^{i}\}$, i.e. the rank of $G_i$ is uniformly bounded.  By Theorem \ref{thm:mainproof} it follows that  $\ker ~ \phi$ is finitely generated. But as already mentioned in the proof of Lemma \ref{lemma:virt}  this can occur  if and only if the HNN extension is {\em simultaneously} ascending and descending, i.e. both $A_{\pm} = B$.
\end{proof}

\section{The Groups $F(\ell,\Z[\frac{1}{n_1\cdots n_k}],\langle n_1,\dots,n_k\rangle)$  are not K\"ahler} \label{stein}

As observed in \cite{NR06,NR08} Corollary \ref{cor:ascending}  implies that Thompson's group $F$ (and some of its generalizations)  is not K\"ahler, as that group is a properly ascending HNN extension. There are several families of generalizations of Thompson's group, that arise as finitely presented subgroups of the group of orientation--preserving PL homeomorphisms of an interval, for which we can expect as well that they fail to be K\"ahler. For ``most" of them, however, Corollary \ref{cor:ascending} is not conclusive: there are examples for which every class $\phi$ in the first integral cohomology group has the property that $\ker~ \phi$ is finitely generated (see below). 

We will be able to decide the problem above by determining, in more generality, some constraints that groups with vanishing rank gradient function  $\rg(G,\phi)$ need to satisfy in order to be K\"ahler.

In order to determine these constraints we will make use of the Bieri--Neumann--Strebel invariant of $G$  (for which we refer the reader to \cite{BNS87}), together with the generalization of \cite[Theorem 4.3]{NR01} due to Delzant.

\begin{proposition}\label{prop:fg} Let $G$ be a K\"ahler group whose rank gradient function $\rg(G,\phi)$ vanishes identically for all $\phi \in H^1(G;\Z)$. Then the commutator subgroup $[G,G] \leq G$ of $G$ is finitely generated. \end{proposition}

\begin{proof} Let $G$ be as in the statement and $X$ a K\"ahler manifold with $G = \pi_1(X)$. Assume, by contradiction, 
 that the commutator subgroup is not finitely generated. By \cite[Theorem B1]{BNS87} there exists an exceptional character $\chi\colon G \to \R$. But then there exists an epimorphism $p\colon G \to \G$ with finitely generated kernel (see \cite{Ca03}) to a cocompact Fuchsian group, induced by the irrational pencil $f\colon X \to \Sigma$ of  \cite[Th\'eor\`eme~1.1]{De10}. The induced map  $p^{*}\colon H^{1}(\G;\Z) \to H^{1}(G\;Z)$ is then injective, and for all $\phi \in \mbox{im}(p^{*}) \subset H^{1}(G;\Z)$ we have $\rg(G,\phi) > 0$ by Lemma  \ref{lemma:fib}, which yields a contradiction to our assumption on $G$. \end{proof} 

There is a case where we can use a simple criterion to certify that a group has rank gradient that vanishes identically  for all $\phi \in H^1(G;\Z)$, and this will be sufficient to apply Proposition \ref{prop:fg} to Thompson's group and its generalization.

\begin{lemma} \label{lem:free} Let $G$ be a finitely presented group  which does not contain any $F_2$ subgroup. Then the rank gradient function  $\rg(G,\phi)$ vanishes identically for all $\phi \in H^1(G;\Z)$. \end{lemma} 

\begin{proof}  Given any primitive class $\phi \in H^{1}(G;\Z)$, by \cite{BS78} there exists a finitely generated subgroup $B \leq \ker ~ \phi \leq  G$ and finitely generated  associated subgroups $A_{\pm} \leq B$ so that $G = \langle B,t| t^{-1}A_{+} t = A_{-} \rangle$, so that $\phi(t) = 1$. It is well--known (see e.g. \cite{BM07} for a detailed proof) that unless this extension is ascending or descending, then $G$ contains an $F_2$ subgroup. If $G$ does not contain any $F_2$ subgroup, repeating the argument of the proof of Corollary \ref{cor:ascending} for all primitive classes $\phi \in H^{1}(G;\Z)$, we deduce the vanishing of $rg(G,\phi)$. \end{proof}

\begin{remark} In light of the results above, we can give a group-theoretic relic of the result of Delzant
(\cite[Th\'eor\`eme~1.1]{De10}): either a K\"ahler group $G$ has finitely generated commutator subgroup, or it can be written as a non--ascending and non--descending HNN extension. This alternative is decided by the vanishing of the rank gradient function.
A straightforward (but not completely trivial) exercise in $3$--manifold group theory shows that this alternative is the same that occurs for a $3$--manifold group $\pi = \pi_1(N)$. (The nontriviality comes from the need to use results on the Thurston norm of  $N$, that has no complete analog in the case of other groups.) For $3$--manifolds with $b_{1}(\pi) \geq 2$, furthermore, the first case corresponds to $N$ a nilmanifold or Euclidean.
For a general group, it is quite possible that the assumption  that the rank gradient function $\rg(G,\phi)$ vanishes identically  does not yield particularly strong consequences (see e.g. the examples discussed in Section \ref{sec:rgbns}).
 \end{remark}

We are now in a position to show that a large class of generalizations of Thompson's group are not K\"ahler.  Omitting further generality, we will limit ourselves to the groups of type $F(\ell,\Z[\frac{1}{n_1\cdots n_k}],\langle n_1,\dots,n_k\rangle)$ discussed by Brown in \cite{Bro87a} and further elucidated by Stein in \cite{St92}. These groups are finitely presented, and can be described as follows: Given a nontrivial multiplicative subgroup \[ \langle n_1,\dots,n_k\rangle = \{n_1^{j_1}\cdots n_k^{j_k}|j_i \in \Z\} \leq \R^{+} \] (where we assume that $\{n_1,\dots,n_{k}\}$ is a positive integer basis  of the multiplicative group), this groups is defined as the group of orientation--preserving PL homeomorphism of the interval $[0,\ell]$, where $0 < \ell \in \Z[\frac{1}{n_1\cdots n_k}]$, with finitely many singularities all in $\Z[\frac{1}{n_1\cdots n_k}]$ and slopes in $\langle n_1,\dots,n_k\rangle$. In this notation, Thompson's group $F$ is $F(1,\Z[\frac12],\langle 2\rangle)$.
The proof that these groups are not K\"ahler is a  consequence of the knowledge of the BNS invariants for the group in question, which follows from the results of  \cite[Section 8]{BNS87}:

\begin{theorem} \label{thm:gnk} Let $G = F(\ell,\Z[\frac{1}{n_1\cdots n_k}],\langle n_1,\dots,n_k\rangle)$; then $G$ is not K\"ahler.  \end{theorem}

\begin{proof} By \cite[Theorem~3.1]{BrSq85} the group $G = F(\ell,\Z[\frac{1}{n_1\cdots n_k}],\langle n_1,\dots,n_k\rangle)$ does not contain $F_2$ subgroups, hence its rank gradient function $\rg(G,\phi)$ identically vanishes. On the other hand, as we show in detail in the  Appendix, 
 \cite[Theorem~8.1]{BNS87} implies that it has exceptional characters, hence its commutator subgroup cannot be finitely generated. By Proposition \ref{prop:fg} $G$ is not K\"ahler. \end{proof}

As observed in \cite[Section 8]{BNS87} ``most" of the groups above have the property that for all primitive classes $\phi \in H^1(G;\Z)$, $\ker ~\phi$ is finitely generated: an explicit example is the group $G = F(1,\Z[\frac{1}{6}],\langle 2,3\rangle)$ of homeomorphisms of the unit interval with slopes in $\{2^j  3^k | j,k \in \Z\}$ and singularities in $\Z[1/6]$; here $H^{1}(G;\Z) = \Z^4$, as computed in \cite{St92}. It follows from this observation that for these groups Corollary \ref{cor:ascending} is inconclusive. A less obvious fact is that we cannot even apply the ``covering trick" to use  Corollary \ref{cor:ascending}, namely fish for a finite index subgroup $H \leq_{f} G$ that is a properly ascending or descending HNN  extension, hence fails to be K\"ahler by Corollary \ref{cor:ascending}:

\begin{lemma}  Let $G = F(\ell,\Z[\frac{1}{n_1\cdots n_k}],\langle n_1,\dots,n_k\rangle)$, and assume that $G$ cannot be written as a properly ascending or descending HNN extension; then neither can any of the finite index subgroups of $G$.\end{lemma}

\begin{proof}  Let $H \leq_{f} G$ be a finite index subgroup, and denote by $f\colon H \to G$ the inclusion map. We start by showing that $f^{*}\colon H^{1}(G;\R) \to H^{1}(H;\R)$ is an isomorphism. By the existence of the transfer map, injectivity is clear, so  we just need to show that $b_{1}(G) = b_{1}(H$). As it suffices to prove this result for the normal core of $H$, we will assume without loss of generality that $H \trianglelefteq_f G$. Then $[H,H]$ is a normal subgroup of $[G,G]$. And as the latter is simple (see e.g. \cite[Theorem~5.1]{St92}, based also on work of Bieri and Strebel) we must have $[H,H]=[G,G]$. The map \[ H_1(H;\Z) = H/[H,H]=H/[G,G] \to G/[G,G] = H_{1}(G;\Z) \] is therefore injective, and the result follows. 
At this point an endless amount of patience would allow a proof of the statement directly in terms of generators and relations, but we can proceed in a simpler way using BNS invariants. In fact, if we assume by contradiction that there exists a primitive class $\chi\colon H \to \Z$ that corresponds to a properly ascending extension, then the (projective class) of $\chi$ must satisfy $[\chi] \in \Sigma(H)$ and  $[\chi] \notin - \Sigma(H)$, where \[ \Sigma(H) \subset S(H) := (H^{1}(H;\R) \setminus 0)/\R^*\] is the BNS invariant of  $H$. Surjectivity of $f^{*}\colon H^{1}(G;\R) \to H^{1}(H;\R)$ entails that $[\chi] \in S(H)$ is the image of some class $[\phi] \in S(G)$, where we can assume that $\phi$ is a primitive class in $H^1(G;\Z)$. But \cite[Theorem H]{BNS87} guarantees now that $f^{*}(S(G) \setminus \Sigma(G)) = S(H) \setminus \Sigma(H)$ (we are  using here the fact that $[H,H] = [G,G]$). It follows that $[\phi] \in \Sigma(G)$ and  $[\phi] \notin - \Sigma(G)$. These two conditions characterize properly ascending HNN extensions (see e.g. \cite[Theorem 2.1]{BGK10}), hence $G$ can be written as an ascending HNN extension. (The proof in the descending case is \textit{verbatim}.) This contradicts the assumption. 
\end{proof}

\section{Rank Gradient and BNS Invariant} \label{sec:rgbns}
While the rank gradient of infinite cyclic covers of $3$--manifold groups and, now, K\"ahler manifolds follow the dichotomy of Theorem \ref{thm:mainproof}, this  property should be counted more of as a property of these very peculiar classes of groups than a common phenomenon. Stated otherwise, the vanishing of the rank gradient $\rg(G,\phi)$ is a necessary but, as we discuss below, not a sufficient condition to decide that $\phi$ belongs to the BNS invariant of $G$.

The first example of this fact comes from Thompson's group $F$. As discussed above (cf. Lemma \ref{lem:free}) this group has vanishing rank gradient for all classes $\phi \in H^1(F;\Z)$, but two pairs of classes $\{\pm \phi_{i},i =1,2\}$ (corresponding to properly ascending HNN extensions) fail to satisfy the statement of Theorem \ref{thm:mainproof}. In this case, $\{\phi_{i},i =1,2\}$ belong to the BNS invariant of  $F$, but $\{- \phi_{i},i =1,2\}$ do not. 

A second, more compelling example is the following, that shows that the discrepancy between the information contained in the rank gradient and in the BNS invariants is not limited to properly ascending or descending extensions. Consider the group
\begin{equation} \label{eq:group} \pi=\ll t,a,b\,|\, t^{-1}at=a^2, t^{-1}b^2t=b\rr.\end{equation}
Here $H_1(\pi;\Z) = \Z$. We claim that $\pi$ cannot be written as a properly ascending or descending HNN extension. To prove this we need the following lemma. Here we denote by  
$\Delta_{\pi,\phi}\in \zt$ the Alexander polynomial corresponding to a pair $(\pi,\phi)$,
i.e. the order of the Alexander $\zt$--module \[ H_1(\pi;\zt)=H_1(\ker(\phi)).\]

\begin{lemma}\label{lem:alex}
Let $\pi$ be a group  together with an epimorphism $\phi\co \pi\to \Z$.
If the pair $(\pi,\phi)$ corresponds to an ascending (resp. descending) HNN extension of a finitely generated group, then
the top (resp. bottom) coefficient of $\Delta_{\pi,\phi}$ equals $\pm 1$.
\end{lemma}

\begin{proof}
Let $\pi$ be a group together with an epimorphism $\phi\co \pi\to \Z=\ll t\rr$.
We denote by $\phi$ also the corresponding ring homomorphism $\Z[\pi]\to \Z[\Z]=\zt$. Assume $\pi$ can be written as an ascending HNN extension of a finitely generated group $B$. In this case it is convenient to write $\pi=\ll t,B\,|\,t^{-1}Bt = \varphi(B)\rr$ where $\varphi\co B\to B$ is a monomorphism, $B = C$ and $D = \varphi(B) \leq B$.
Now let \[ B=\ll g_1,\dots,g_k\,|\, r_1,r_2,\dots\rr \] be a presentation for $B$ with finitely many generators. We denote by $n\in \N\cup \{\infty\}$ the number of relations.
We can write 
\[ \pi=\ll t,g_1,\dots,g_k\,|\, t^{-1}g_1t=\varphi(g_1),\dots,t^{-1}g_kt=\varphi(g_k), r_1,r_2,\dots\rr.\]
We continue with the following claim.

\begin{claim}
The matrix
\[ \left(  \phi\left(\frac{\partial t^{-1}g_it}{\partial g_j}\right)_{i,j=1,\dots,k}\,  \phi\left(\frac{\partial r_i}{\partial g_j}\right)_{i=1,\dots,j=1,\dots,k} \right)\]
is a presentation matrix for $H_1(\pi;\zt)$.
\end{claim}

We denote by $X$ the usual CW-complex corresponding to the presentation for $\pi$ and we denote by $Y$ the CW-complex corresponding to the presentation $\ll t\rr$. 
By Fox calculus the chain complex $C_*(X;\zt)$ is given as follows.
\[  \zt^k\oplus \zt^n\,\xrightarrow{\bp  \phi\left(\frac{\partial t^{-1}g_it}{\partial g_j}\right)&  \phi\left(\frac{\partial r_i}{\partial g_j}\right)\\[2mm]
\phi\left(\frac{\partial t^{-1}g_it}{\partial t}\right)&  \phi\left(\frac{\partial r_i}{\partial t}\right)\ep}\zt^k\oplus \zt
\xrightarrow{\big(1-x_1\,\,\dots\,\,1-x_k\,\,1-t\big)}\zt\to 0.\]
Here, for space reasons we left out the range of the indices $i$ and $j$.
The chain complex 
\[ C_*(Y;\zt)\,\,=\,\,0\to \zt\xrightarrow{1-t}\zt\to 0\]
 is an obvious subcomplex of $C_*(X;\zt)$.
It is now straightforward to see that the matrix in the claim is a presentation matrix 
for the relative Alexander module 
\[ H_1(\pi,\ll t\rr;\zt)=H_1\Big(C_*(X;\zt)/C_*(Y;\zt)\Big).\]
But considering the long exact sequence in homology with $\zt$-coefficients of the pair $(\pi,\ll t\rr)$ shows that
$H_1(\pi,\ll t\rr;\zt)\cong H_1(\pi;\zt)$. This concludes the proof of the claim.

We now recall that  $\zt$ is a Noetherian ring. This implies that every ascending chain of submodules of a given finitely generated module stabilizes.
This  in turn implies that every finitely generated module over $\zt$ which is given by infinitely many relations is already described by a finite subset of those relations. (We refer to \cite[Chapter~X.\S 1]{Lan02} for details.)

It thus follows that there exists an $l\in \N$ such that the following matrix
is already a presentation matrix for $H_1(\pi;\zt)$.
\[ \left(  \phi\left(\frac{\partial t^{-1}g_it}{\partial g_j}\right)_{i,j=1,\dots,k}\,  \phi\left(\frac{\partial r_i}{\partial g_j}\right)_{i=1,\dots,l,j=1,\dots,k} \right).\]

The block matrix on the right is a matrix over $\Z$, whereas the block matrix on the left is a matrix of the form $t\id_k-B$ where $B$ is a matrix over $\Z$. If we now take the gcd of the $k\times k$-minors of the matrix we see that the top coefficient is $1$. The proof for the descending case is identical. \end{proof}

We have now the following:

\begin{proposition} Let $\pi$ be the finitely presented group of Eq. (\ref{eq:group}) and let $\phi \in H^1(\pi;\Z) \cong \Z$ be the generator given by the epimorphism of $Hom(\pi,\Z)$ determined  by $\phi(t)=1$ and $\phi(a)=\phi(b)=0$.
The pair $(\pi,\phi)$ has zero rank gradient but it cannot be written as an ascending or descending HNN extension.
\end{proposition}

\begin{proof}
We denote by $\pi_n$ the kernel of the epimorphism $\pi\xrightarrow{\phi}\Z\to \Z/n$. It follows from the Reidemeister-Schreier process, that $\pi_n$ has a presentation of the form
\[ \ba{rl} \ll s,a_0,\dots,a_{n-1},b_0,\dots,b_{n-1}|&a_1=a_0^2,\dots,a_{n-1}=a_{n-2}^2,s^{-1}a_0s=a_{n-1}^2,\\ &
b_1^2=b_0,\dots,b_{n-1}^2=b_{n-2},s^{-1}b_0^2s=b_{n-1}^2\rr\end{array}\]
Here, the map $\pi_n\to \pi$ is given by sending $s$ to $t^n$ and sending
$a_i\mapsto t^{-i}at^i$ and $b_i\mapsto t^{-i}bt^i$.
The fact that this is a correct presentation for $\pi_n$ can also be seen easily by thinking of $\pi$ as the fundamental group of a 2-complex with one 0-cell, three 1-cells and two 2-cells.

Note that in the presentation of $\pi_n$ we can remove the generators $a_1,\dots,a_{n-1}$ and $b_1,\dots,b_{n-1}$. We can thus find a presentation for $\pi_n$ with three generators. It is now immediate that $\rg(\pi,\phi)$ is zero.
\\

A straightforward calculation using Fox calculus shows that 
the  Alexander polynomial of  $(\pi,\phi)$ is $(t-2)(2t-1)=2-5t+2t^2$.
It now follows from Lemma \ref{lem:alex} that the pair $(\pi,\phi)$  corresponds neither to an ascending nor descending HNN extension. 
\end{proof}
\appendix
\setcounter{secnumdepth}{0}
\section{Appendix}
We include in this appendix definitions and results needed in the proof of Theorem \ref{thm:gnk}. The purpose is to show that  the class of groups discussed in Section \ref{stein} satisfies the assumptions of Theorem 8.1 of \cite{BNS87}, which thus determines their BNS invariant. This is arguably standard fare for the practitioners of Thompson--type groups, but may be  less obvious for others, so we provide here the details.

Let $G$ be a subgroup of the group of orientation--preserving PL homeomorphisms of an interval $[0,\ell]$.  We say that such a group is \textit{irreducible} if it has no fixed point in $(0,\ell)$. Given such a group, we can define two homomorphisms $\lambda, \rho\colon G \to \R_{+}$ with values in the multiplicative group $ \R_{+}$ as follows: \[ \lambda(g) := \frac{dg}{dt}(0), \,\  \rho(g) := \frac{dg}{dt}(\ell), \] which return the slope at the endpoints. We say that $\lambda$ and $\rho$ are \textit{independent} if $\lambda(G)= \lambda (\Ker \rho)$ and $\rho (G)= \rho (\Ker \lambda)$. Equivalently, \[ G \xrightarrow{(\rho,\lambda)} \lambda(G) \times \rho(G) \] is an epimorphism. Then we have the following theorem (see \cite[Theorem~8.1]{BNS87}, as well as \cite[Proposition~8.2]{Bro87b}).

\begin{theorem*} (Bieri--Neumann--Strebel) Let $G$ be a finitely generated, irreducible subgroup of the group of orientation--preserving PL homeomorphisms of an interval $[0,\ell]$. Then if $\lambda$ and $\rho$ are independent, the set of exceptional characters of $G$ is given by $[\log(\lambda)],[\log(\rho)] \in \hom(G,\R) = H^{1}(G;\R)$.  \end{theorem*}
(The statement in  \cite{BNS87} covers subgroups of the group of orientation--preserving PL homeomorphisms of the unit interval $[0,1]$, but the proof applies with obvious modifications to the case stated here.)

We will apply this result to the study of the BNS invariants for the case of the generalized Thompson groups $F(\ell,\Z[\frac{1}{n_1\dots n_k}],\langle n_1,\dots,n_k\rangle)$ which, we recall, are defined as follows. Given a nontrivial multiplicative subgroup \[ \langle n_1,\dots,n_k\rangle = \{n_1^{j_1}\dots n_k^{j_k}|j_i \in \Z\} \leq \R^{+} \] (where $\{n_1,\dots,n_{k}\}$ is a positive integer basis  of the multiplicative group), we denote by  $F(\ell,\Z[\frac{1}{n_1\dots n_k}],\langle n_1,\dots,n_k\rangle)$ the group of orientation--preserving PL homeomorphism of the interval $[0,\ell]$, where
 $0 < \ell \in \Z[\frac{1}{n_1\dots n_k}]$, with finitely many singularities all in $\Z[\frac{1}{n_1\cdots n_k}]$ and slopes in $\langle n_1,\dots,n_k\rangle$.  
\begin{corollary*} Let   $G = F(\ell,\Z[\frac{1}{n_1\cdots n_k}],\langle n_1,\dots,n_k\rangle)$ be a generalized Thompson group; then $G$ admits exceptional characters. \end{corollary*}
\begin{proof} The proof of the corollary consists in verifying that the assumptions of the Bieri--Neumann-Strebel Theorem are satisfied. First, we show that the group $G$ is irreducible: we can do so by constructing an explicit element $g \in G$ such that $g\colon [0,\ell] \to [0,\ell]$ has no other fixed points but $\{0,\ell\}$. To do so, choose an element $\nu \in \langle n_1,\dots,n_k\rangle$ such that $\nu > 1$,  and  define $n := \prod_{i=1}^{k} n_{i} \geq 2$. Consider the pairs $(x_1,y_1),(x_2,y_2) \subset [0,\ell] \times [0,\ell]$ with coordinates \[ (x_1,y_1)=\Big( \lmfrac{1}{\nu} \lmfrac{\ell}{n},\lmfrac{\ell}{n}\Big), \,\  (x_2,y_2)= \Big(\ell -  \lmfrac{\ell}{n},\ell - \lmfrac{1}{\nu} \lmfrac{\ell}{n}\Big). \] It is immediate to verify that as $\ell \in \Z[\frac{1}{n}]$, then $x_1 < x_{2} \in (0,\ell) \cap \Z[\frac{1}{n}]$. The piecewise linear graph joining the four points $(0,0),(x_1,y_1),(x_2,y_2),(\ell,\ell) \in [0,\ell] \times [0,\ell]$ is composed of three segments of slope respectively $\nu,1,\frac{1}{\nu} \in \langle n_1,\cdots,n_k\rangle$. Therefore, this represents the graph of an element $g \in G$ that has no fixed points in the interval $(0,\ell)$, hence the proof that $G$ is irreducible. On the left hand side of Figure \ref{tikz} we show the graph of such an element for the generalized Thompson group $ F(1,\Z[\frac{1}{6}],\langle 2,3\rangle)$, with $\nu = 2$.

Next, let's show that $\lambda$ and $\rho$ are independent. We will limit ourselves to show that $\lambda(G)= \lambda (\Ker \rho)$, the proof of the symmetric relation being identical. Namely, we will show that for any element $\nu  \in \langle n_1,\dots,n_k\rangle$ there exists an element $h \in \Ker \rho$ (i.e. for which $\rho(h) = \frac{dh}{dt}(l) = 1$) such that $\frac{dh}{dt}(0) = \nu$. The proof is based on a suitable scaled version of the construction above. Define ${\tilde \ell} = \ell - \frac{\ell}{n}$, and assume first $\nu > 1$. We consider now the pairs  \[ (x_1,y_1),(x_2,y_2) \subset [0,{\tilde \ell}] \times [0,{\tilde \ell}] \subset [0,\ell] \times [0,\ell] \] with coordinates 
\[ (x_1,y_1)= \Big( \lmfrac{1}{\nu} \lmfrac{\tilde \ell}{n},\lmfrac{\tilde \ell}{n}\Big), \,\  (x_2,y_2)= \Big({\tilde \ell} -  \lmfrac{\tilde \ell}{n},{\tilde \ell} - \lmfrac{1}{\nu} \lmfrac{\tilde \ell}{n}\Big). \] Similar to above, it is immediate to verify that  $x_1 < x_{2} < {\tilde \ell} \in (0,\ell) \cap \Z[\frac{1}{n}]$. The piecewise linear graph joining the five points $(0,0),(x_1,y_1),(x_2,y_2),({\tilde \ell},{\tilde \ell}),(\ell,\ell) \in [0,\ell] \times [0,\ell]$ is composed of four segments of slope respectively $\nu,1,\frac{1}{\nu},1 \in \langle n_1,\dots,n_k\rangle$. It follows that this represents the graph of an element $h \in G$ such that $\lambda(h) = \nu$ and $\rho(h) = 1$.  The right hand side of Figure \ref{tikz} show the graph of $h$ for the generalized Thompson group $F(1,\Z[\frac{1}{6}],\langle 2,3\rangle)$, with $\nu = 2$.

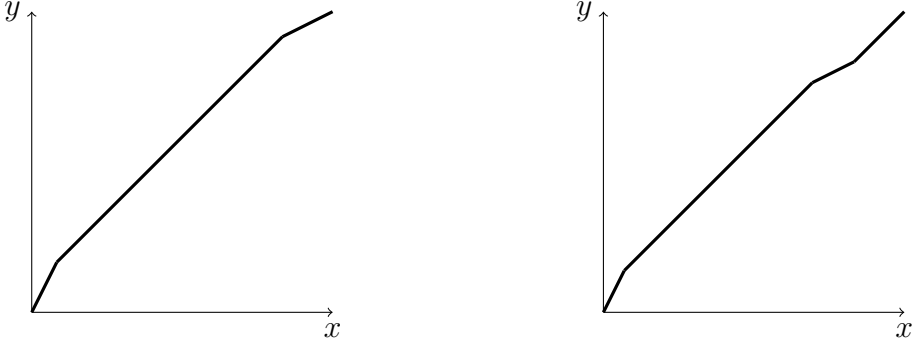
\begin{figure}[h]
\begin{tikzpicture}[xscale=4,yscale=4] \label{Figure}
\draw[shift={+(0.1,0)}][->] (0,0) to (1,0) node[below]{$x$};
\draw[shift={+(0.1,0)}][->] (0,0) to (0,1) node[left]{$y$};
\draw[shift={+(0.1,0)}][very thick, domain=0:1/12] plot (\x,2*\x);
\draw[shift={+(0.1,0)}][very thick, domain=1/12:5/6] plot (\x,\x + 1/12);
\draw[shift={+(0.1,0)}][very thick, domain=5/6:1] plot (\x,1/2*\x+1/2);

\draw[shift={+(0.75,-0.3)}] (0,0) node {$g \in F(1,\Z[\frac{1}{6}],\langle 2,3\rangle)$; $\lambda(g) =2,\rho(g) = \frac{1}{2}$};
\draw[shift={+(2.65,-0.3)}] (0,0) node {$h \in F(1,\Z[\frac{1}{6}],\langle 2,3\rangle)$; $\lambda(h) =2,\rho(h) = 1$};

\draw[shift={+(2,0)}][->] (0,0) to (1,0) node[below]{$x$};
\draw[shift={+(2,0)}][->] (0,0) to (0,1) node[left]{$y$};
\draw[shift={+(2,0)}][very thick, domain=0:5/72] plot (\x,2*\x);
\draw[shift={+(2,0)}][very thick, domain=5/72:25/36] plot (\x,\x + 5/72);
\draw[shift={+(2,0)}][very thick, domain=25/36:5/6] plot (\x,1/2*\x+5/12);
\draw[shift={+(2,0)}][very thick, domain=5/6:1] plot (\x,\x);

\end{tikzpicture}
\captionof{figure}{}
   \label{tikz}
\end{figure}

The proof for the case of $\nu < 1$ is similar, defining this time 
\[ (x_1,y_1)= \Big( \lmfrac{\ell}{n}, \nu \lmfrac{\ell}{n}\Big), \,\  (x_2,y_2)= \Big(\ell -  \nu \lmfrac{\ell}{n},\ell - \lmfrac{\ell}{n}\Big) \] and performing similar elementary calculations. It now follows from the theorem of Bieri--Neumann--Strebel that the set of exceptional characters is nonempty as claimed in the statement.
\end{proof}



\begin{thebibliography}{10}
\bibitem[ABCKT96]{ABC$^+$96} J. Amor\'os, M.  Burger,  K. Corlette, D. Kotschick  and D. Toledo, {\em Fundamental groups of compact K\"ahler manifolds},
Mathematical Surveys and Monographs, 44. American Mathematical Society, Providence, RI, 1996.
\bibitem[Ag13]{Ag13} I. Agol, {\em The virtual Haken conjecture}, with an appendix by I. Agol, D. Groves and J.
Manning,  Documenta Math. 18 (2013) 1045--1087.
\bibitem[BM07]{BM07} G. Baumslag and C. F. Miller III, {\em Finitely presented extensions by free groups}, J. Group Theory 10 (2007), 723--729.
\bibitem[BNS87]{BNS87} R. Bieri, W. Neumann and  R. Strebel, {\em A geometric invariant of discrete groups}, Invent. Math. 90 (1987), no. 3, 451--477.
\bibitem[BGK10]{BGK10} R. Bieri, R. Geoghegan and D. Kochloukova, {\em The sigma invariant of Thompson's group $F$}, Groups Geom. Dyn. 4  no 2 (2010), 263--273.
\bibitem[BS78]{BS78}  R. Bieri and R. Strebel, {\em Almost finitely presented soluble groups}, Comment. Math. Helv. 53
(1978), 258--278.
\bibitem[BrSq85]{BrSq85} M. Brin and C. Squier, {\em Groups of piecewise linear homeomorphisms of the real line},  Invent. Math. 79 (1985), no. 3, 485--498.
\bibitem[Bro87a]{Bro87a} K. S. Brown, {\em Finiteness properties of groups}, 
Proceedings of the Northwestern conference on cohomology of groups (Evanston, Ill., 1985), J. Pure Appl. Algebra 44 (1987), no. 1-3, 45--75.
\bibitem[Bro87b]{Bro87b} K. S. Brown, {\em Trees, valuations, and the Bieri--Neumann--Strebel invariant},  Invent. Math. 90 (1987), no. 3, 479--504.
\bibitem[Bu11]{Bu11} M. Burger, {\em Fundamental groups of K\"ahler manifolds and geometric group theory}, S\'eminaire Bourbaki 62 (2009/2010). Expos\'e 1022. Ast\'erisque No. 339 (2011)
\bibitem[Ca03]{Ca03} F. Catanese, {\em Fibred K\"ahler and quasi--projective groups},  Special issue dedicated to Adriano Barlotti, Adv. Geom. suppl. (2003), 13--27.
\bibitem[De10]{De10} T. Delzant, {\em L'invariant de Bieri Neumann Strebel des groupes fondamentaux des vari\'et\'es k\"ahleriennes}, Math. Annalen 348 (2010), 119--125.
\bibitem[DFV14]{DFV14} J. DeBlois, S. Friedl, and S. Vidussi, {\em Rank gradients of infinite cyclic covers of 3-manifolds}, Michigan Math. J.  63 (2014), no. 1, 65--81.
\bibitem[Lan02]{Lan02}
S. Lang, {\em Algebra}, third  revised edition,
Graduate Texts in Mathematics. 211. (2002)
\bibitem[Lac05 ]{Lac05 } M. Lackenby, {\em Expanders, rank and graphs of groups}, Israel J. Math., 146 (2005), 357--370.
\bibitem[Lu96]{Lu96} A. Lubotzky, {\em Free quotients and the first Betti number of some hyperbolic manifolds}, Transform. Groups 1, no. 1-2 (1996), 71--82.
\bibitem[NR01]{NR01} T. Napier  and M. Ramachandran, {\em Hyperbolic K\"ahler manifolds and proper holomorphic mappings to Riemann surfaces}, Geom. Funct. Anal. 11, no 2 (2001), 382--406.
\bibitem[NR06]{NR06} T. Napier and M. Ramachandran, {\em Thompson's group $ F$ is not K\"ahler}, Topological and asymptotic aspects of group theory, 197--201, Contemp. Math., 394, Amer. Math. Soc., Providence, RI, 2006.
\bibitem[NR08]{NR08}  T. Napier, and M. Ramachandran, {\em Filtered ends, proper holomorphic mappings of K\"ahler manifolds to Riemann surfaces, and K\"ahler groups}, Geom. Funct. Anal. 17, no 5 (2008), 1621--1654.
\bibitem[Sc78]{Sc78} P. Scott, {\em Subgroups of surface groups are almost geometric},
J. London Math. Soc. (2) 17 no. 3 (1978), 555--565.
\bibitem[Se80]{Se80} J.-P. Serre, {\em Trees}, Springer-Verlag, Berlin-New York, 1980.
\bibitem[Sta62]{S62}
J. Stallings, {\em On fibering certain 3--manifolds}, 1962 Topology
of 3--manifolds and related topics (Proc. The Univ. of Georgia
Institute, 1961) pp. 95--100 Prentice-Hall, Englewood Cliffs, N.J. (1962)
\bibitem[Ste92]{St92} M. Stein, {\em  Groups of piecewise linear homeomorphisms}, Trans. Amer. Math. Soc. 332 (1992), no. 2, 477--514.
\bibitem[Wi12]{Wi12} D. Wise, {\em The structure of groups with a quasi-convex hierarchy}, 189 pages, preprint
(2012), downloaded on October 29, 2012 from
http://www.math.mcgill.ca/wise/papers.html
\end{thebibliography}
\end{document}